\documentclass{amsart}

\usepackage{txfonts}
\newtheorem{theorem}{Theorem}[section]
\newtheorem{proposition}[theorem]{Proposition}
\newtheorem{lemma}[theorem]{Lemma}

\newtheorem{corollary}[theorem]{Corollary}

\theoremstyle{definition}
\newtheorem{definition}[theorem]{Definition}
\newtheorem{example}[theorem]{Example}

\theoremstyle{remark}

\numberwithin{equation}{section}

\catcode`\ç=13
\defç{\c{c}}
\catcode`\é=13
\defé{\'e}
\catcode`\à=13
\defà{\`a}
\catcode`\è=13
\defè{\`e}
\catcode`\â=13
\defâ{\^a}
\catcode`\ù=13
\defù{\`u}
\catcode`\ê=13
\defê{\^e}
\catcode`\î=13
\defî{\^\i}
\catcode`\ô=13
\defô{\^o}
\def\Ker{{\rm Ker}}

\begin{document}

\title{Gaussian property in amalgamated algebras along an ideal}

\author[Najib Mahdou]{Najib Mahdou}
\address{Department of Mathematics, Faculty of Science and Technology of Fez, Box 2202, University S. M.
Ben Abdellah Fez, Morocco}
 \email{mahdou@hotmail.com}

%\curraddr{Department of Mathematics, Faculty of Science and
%Technology of Fez,\\ Box 2202, University S. M. Ben Abdellah Fez,
%Morocco}
 \email{moutu\_2004@yahoo.fr}

\author[Moutu Abdou Salam Moutui]{Moutu Abdou Salam Moutui}


%    General info

%\date{January 1, 2001 and, in revised form, June 22, 2001.}

%\dedicatory{Dedicated to our Advisor Salah-Eddine Kabbaj.}


\begin{abstract} Let $f : A \rightarrow B$ be a ring homomorphism and  $J$ be an ideal of $B$. In this paper, we investigate the transfer of Gaussian
property to the amalgamation of $A$ with $B$ along $J$ with respect to $f$
(denoted by $A\bowtie^fJ),$ introduced and studied by D'Anna, Finocchiaro and Fontana in 2009. \smallskip
\end{abstract}

\maketitle

\section{Introduction} All rings considered in this paper are commutative with identity elements and all modules are unital. In 1932, Pr\"ufer introduced and studied in \cite{P} integral domains in which every finitely generated ideal is invertible. In their recent paper devoted to Gaussian properties, Bazzoni and Glaz have proved that a Pr\"ufer ring satisfies any of the other four Pr\"ufer conditions if and only if its total ring of quotients satisfies that same condition \cite[Theorems 3.3 \& 3.6 \& 3.7 \& 3.12]{BG2}. Vasconcelos regarding a conjecture of Kaplansky: The content ideal of a Gaussian polynomial is an invertible (or locally principal) ideal. The reason behind the conjecture is that the converse holds \cite{T}. Vasconcelos  and  Glaz answered the question affirmatively in a large number of cases \cite{GV,gl}. The affirmative answer was later extended by Heinzer and Huneke \cite{HH} to include all Noetherian domains. Recently the question was answered affirmatively for all domains by Loper and Roitman \cite{LR}, and finally to non-domains provided the content ideal has zero annihilator by Lucas \cite{Lu}. The article by Corso and Glaz \cite{c} gives a good account of what was known about the problem prior to the year 2000 or so. At the October 2003 meeting of the American Mathematical Society in Chapel Hill and North Carolina, Loper presented a proof that every Gaussian polynomial over an integral domain has invertible content (in fact, for any ring that is locally an integral domain). The basis of the proof presented in this paper is highly dependent on the work of Loper and his coauthor Roitman. A flurry of related research ensued, particularly investigations involving Dedekind-Mertens Lemma and various extensions of the Gaussian property. \cite{and1} and \cite{c} provide a survey of results obtained up to year 2000 and an extensive bibliography. A related, but different, question is : How Pr\"ufer-like is a Gaussian
ring? Various aspects of the nature of Gaussian rings were investigated in Tsang's thesis \cite{T}, Anderson and Camillo \cite{and3}, and Glaz \cite{G2}. While all of those works touch indirectly on the mentioned question, it is Glaz  \cite{{G2}} that asks and provides some direct answers. A problem initially associated with Kaplansky and his student Tsang \cite{BG,GV,Lu,T} and also termed as Tsang-Glaz-Vasconcelos conjecture in \cite{HH} sustained that every nonzero Gaussian polynomial over a domain has an invertible (or, equivalently, locally principal) content ideal." It is well-known that a polynomial over any ring is Gaussian if its content ideal is locally principal. The converse is precisely the object of Kaplansky-Tsang-Glaz-Vasconcelos conjecture extended to those rings where every Gaussian polynomial has locally principal content ideal. In \cite{bkm}, the authors examined the transfer of the Pr\"ufer conditions and obtained further evidence for the validity of Bazzoni-Glaz conjecture sustaining that "the weak global dimension of a Gaussian ring is 0, 1, or $\infty$" \cite{BG2}. Notice that both conjectures share the common context of rings. They gave new examples of non-arithmetical Gaussian rings as well as arithmetical rings with weak global dimension strictly greater than one. Abuihlail, Jarrar and Kabbaj studied in \cite{abjk} the multiplicative ideal structure of commutative rings in which every finitely generated ideal is quasi-projective. They provide some preliminaries quasi-projective modules over commutative rings and they investigate the correlation with well-known Pr\"ufer conditions; namely, they proved that this class of rings stands strictly between the two classes of arithmetical rings and Gaussian rings. Thereby, they generalized Osofsky's theorem on the weak global dimension of arithmetical rings and partially resolve Bazzoni-Glaz's related conjecture on Gaussian rings. They also established an analogue of Bazzoni-Glaz results on the transfer of Pr\"ufer conditions between a ring and its total ring of quotients. In \cite{CJKM}, the authors studied the transfer of the notions of local Pr\"ufer ring and total ring of quotients. They examined the arithmetical, Gaussian, and fqp conditions to amalgameted duplication along an ideal. They also investigated the weak global dimension of an amalgamation and its possible inheritance of the semihereditary condition. At this point, we make the following definition:
\begin{definition}  Let $R$ be a commutative ring.\\
\begin{enumerate}
\item $R$ is called an \emph{arithmetical ring} if the lattice formed by its ideals is distributive (see \cite{Fu}). \\
\item $R$ is called a \emph{Gaussian ring} if for every $f, g \in R[X]$, one has the content ideal equation $c(fg) = c(f)c(g)$ (see \cite{T}).\\
\item $R$ is called a \emph{$Pr\ddot{u}fer$ ring} if every finitely generated regular ideal of $R$ is invertible (see \cite{BS,Gr}).
\end{enumerate}
\end{definition}
In the domain context, all these forms coincide with the definition of a $Pr\ddot{u}fer$ domain. Glaz \cite{G3} provides
examples which show that all these notions are distinct in the context of arbitrary rings. The following diagram of implications
 summarizes the relations between them \cite{BG,BG2,G2,G3,LR,Lu,T}:
 \begin{center}
 Arithmetical $\Rightarrow$ Gaussian $\Rightarrow$ $Pr\ddot{u}fer$
\end{center}
and examples are given in \cite{G3} to show that, in general, the implications cannot be reversed.
Arithmetical and Gaussian notions are local; i.e., a ring is arithmetical (resp., Gaussian) if
and only if its localizations with respect to maximal ideals are arithmetical (resp., Gaussian).
We will make frequent use of an important characterization of a local Gaussian ring;
namely, "for any two elements $a,b$ in the ring, we have $<a,b>^2=<a^2>$ or $<b^2>$; moreover, if
$ab=0$ and, say, $<a,b>^2=<a^2>$, then $b^2=0$" (by \cite[Theorem 2.2 ]{BG2}).\\

In this paper, we study the transfer of Gaussian property in
amalgamation of rings issued from local rings, introduced and
studied by D'Anna, Finocchiaro and Fontana in \cite{AFF1,AFF2} and
defined as follows :

\begin{definition}

Let $A$ and $B$ be two rings with unity, let $J$ be an ideal of
$B$ and let $f: A\rightarrow B$ be a ring homomorphism. In this
setting, we can consider the following subring of $A\times B$:
\begin{center} $A\bowtie^{f}J: =\{(a,f(a)+j)\mid a\in A,j\in
J\}$\end{center} called \emph{the amalgamation of $A$ and $B$
along $J$ with respect to $f$}. In particular, they have studied amalgmations in the frame of pullbacks which allowed them to establish
numerous (prime) ideal and ring-theoretic basic properties for this new construction. This
construction is a generalization of \emph{the amalgamated
duplication of a ring along an ideal} (introduced and studied by
D'Anna and Fontana in \cite{A, AF1, AF2}). The interest of amalgamation
resides, partly, in its ability to cover several basic constructions in commutative algebra,
including pullbacks and trivial ring extensions (also called Nagata's idealizations) (cf. \cite[page 2]{Nagata}). Moreover, other
classical constructions (such as the $A+XB[X]$, $A+XB[[X]]$ and the $D+M$ constructions) can be studied as particular cases of the amalgamation (\cite[Examples 2.5 and 2.6]{AFF1}) and other classical constructions, such as the CPI extensions (in the sense of Boisen and Sheldon \cite{Boisen}) are strictly related to it (\cite[Example 2.7 and Remark 2.8]{AFF1}). In \cite{AFF1}, the authors studied the basic properties of this construction (e.g., characterizations for $A\bowtie^{f}J$ to be a Noetherian ring, an integral domain, a reduced ring) and they characterized those distinguished pullbacks that can be expressed as an amalgamation. Moreover, in \cite{AFF2}, they pursued the investigation on the structure of the rings of the form $A\bowtie^{f}J$, with particular attention to the prime spectrum, to the chain properties and to the Krull dimension.
\end{definition}

\section{Transfer of Gaussian property in amalgamated algebras along an ideal}\label{sec:2}

\bigskip

The main Theorem of this paper develops a result on the transfer of the Gaussian property to amalgamation of rings issued from local rings.
\begin{theorem}\label{thm0} Let $(A,m)$ be a local ring, $B$ be a ring, $f : A \rightarrow B$ be a ring homomorphism and  $J$ be a  proper ideal of
$B$ such that $J\subseteq Rad(B)$. Then, the following statements hold :\\
\begin{enumerate}
\item If $A\bowtie^{f}J$ is Gaussian, then so are $A$ and $f(A)+J.$

\item Assume that $J^2=0.$ Then $A\bowtie^{f}J$ is Gaussian if and only if so is $A$ and $f(a)J=f(a)^2J$ $\forall$ $ a\in m.$
\item Assume that $f$ is injective. Then  two cases are possible: \\
                                 Case 1 : $f(A)\cap J=(0).$ Then $A\bowtie^{f}J$ is Gaussian if and only if so is $f(A)+J$.\\
                                 Case 2 : $f(A)\cap J\neq(0)$. Assume that $A$ is reduced. Then $A\bowtie^{f}J$ is Gaussian if and only if so is $A,$ $J^2=0$ and $f(a)J=f(a)^2J$ $\forall$ $ a\in m.$
\item Assume that $f$ is not injective. Then two cases are possible: \\
                                 Case 1 : $J\cap Nilp(B)=(0)$. If $A$ is reduced, then $A\bowtie^{f}J$ is not Gaussian.\\
                                 Case 2 : $J\cap Nilp(B)\neq(0).$ Assume that $A$ is reduced. Then $A\bowtie^{f}J$ is Gaussian if and only if so is $A$, $J^2=0$ and $f(a)J=f(a)^2J$ $\forall$ $a\in m$.
                                \end{enumerate}
\end{theorem}
Before proving Theorem \ref{thm0}, we establish the following Lemma.
\begin{lemma}\label{lem1}
Let $(A,B)$ be a pair of rings, $f : A \rightarrow B$ be a ring homomorphism and  $J$ be a proper ideal of
$B$. Then, $A\bowtie^{f}J$ is local if and only if so is $A$ and $J\subseteq Rad(B)$.
\end{lemma}
\begin{proof}
By \cite[Proposition 2.6 (5)]{AFF2}, $Max(A\bowtie^{f}J)=\{m\bowtie^{f}J$ / $m\in Max(A)\}\cup\{\overline{Q}^f\}$ with $Q\in Max(B)$ not containing $V(J)$
and $\overline{Q}^f:=\{(a,f(a)+j)$ / $a\in A, j\in J$ and $f(a)+j \in Q$ \}.
Assume that $A\bowtie^{f}J$ is local. It is clear that $A$ is local by the above characterization of $Max(A\bowtie^{f}J)$. We claim that
$J\subseteq Rad(B)$. Deny. Then there exist $Q\in Max(B)$ not containing $V(J)$ and so $Max(A\bowtie^{f}J)$ contains at least two maximal
ideals, a contradiction since $A\bowtie^{f}J$ is local. Hence, $J\subseteq Rad(B)$. Conversely, assume that $(A,m)$ is local and $J\subseteq Rad(B)$.
Then $J$ is contained in $Q$ for all $Q\in Max(B)$. Consequently, the set $\{\overline{Q}^f\}$ is empty. And so
$Max(A\bowtie^{f}J)=\{m\bowtie^{f}J/m\in Max(A)\}$. Hence, $m\bowtie^{f}J$ is the only maximal ideal of $A\bowtie^{f}J$ since $(A,m)$ is local.
Thus, $(A\bowtie^{f}J,M)$ is local with $M=m\bowtie^{f}J$, as desired.\end{proof}

\begin{proof}[Proof of Theorem \ref{thm0}]By Lemma \ref{lem1}, $(A\bowtie^{f}J,m\bowtie^{f}J)$ is local since $(A,m)$ is local and $J\subseteq Rad(B)$.\\
$(1)$ If $A\bowtie^{f}J$ is Gaussian, then $A$ and $f(A)+J$ are Gaussian rings since the Gaussian property is stable under factor ring ( here $A\cong \frac{A\bowtie^{f}J}{\{0\}\times \{J\}}$ and $f(A)+J\cong \frac{A\bowtie^{f}J}{f^{-1}(J)\times \{0\}},$ by \cite[Proposition 5.1 (3)]{AFF1}).\\
$(2)$ Assume that $J^2=0$. \\
 If $A\bowtie^{f}J$ is Gaussian, then so is $A$ by $(1)$. We claim that $f(a)J=f(a)^2J$ for all $a\in m.$ Indeed, it is clear that $f(a)^2J\subseteq f(a)J.$ Conversely, let $x\in J$ and let $a\in m$. Clearly, $(a,f(a))$ and $(0,x)$  are elements of $ A\bowtie^{f}J$. And so $<(a,f(a)),(0,x)>^2=<(a,f(a))^2>$ since $J^2=0$. It follows that $xf(a)=jf(a)^2$ for some $j\in J$. Hence, $f(a)J=f(a)^2J$. Conversely, suppose $A$ is Gaussian and $f(a)J=f(a)^2J$ for all $ a\in m$. We claim that $A\bowtie^{f}J$ is Gaussian. Indeed, Let $(a,f(a)+i)$ and $(b,f(b)+j)\in A\bowtie^{f}J$. Then $a$ and $b\in A$. We may assume that $a,b\in m$ and $<a,b>^2=<a^2>$. Therefore, $b^2=a^2x$ and $ab=a^2y$ for some $x,y\in A.$ Moreover $ab=0$ implies that $b^2=0$. By assumption, there exist $j_1,i_1,j_2,i_2,i_3\in J$ such that $2f(b)j=f(a)^2f(x)j_1,$ $2f(a)if(x)=f(a)^2i_1,$ $f(a)j=f(a)^2j_2$, $f(b)i=f(a)^2f(x)i_2$ and $2f(a)if(y)=f(a)^2i_3$. In view of the fact $J^2=0,$ one can easily check that $(b,f(b)+j)^2=(a,f(a)+i)^2(x,f(x)+f(x)j_1-i_1)$ and $(a,f(a)+i)(b,f(b)+j)=(a,f(a)+i)^2(y,f(y)+f(x)i_2+j_2-i_3)$. Moreover, assume that $(a,f(a)+i)(b,f(b)+j)=(0,0)$. Hence, $ab=0$ and so $b^2=0.$ Consequently, $(b,f(b)+j)^2=(0,0)$. Finally, $A\bowtie^{f}J$ is Gaussian.\\
$(3)$ Assume that $f$ is injective.\\
Case 1 : Suppose that $f(A)\cap J=(0)$. If $A\bowtie^{f}J$ is Gaussian, then so is $f(A)+J$ by $(1)$. Conversely, assume that $f(A)+J$ is Gaussian. We claim that the natural projection :\\
$p:A\bowtie^fJ \rightarrow f(A)+J$ \\
$p((a,f(a)+j))=f(a)+j$ is a ring isomorphism. Indeed, it is clear that $p$ is surjective. It remains to show that $p$ is injective. Let $(a,f(a)+j)$ $\in$ $Ker(p),$ it is clear that $f(a)+j=0$. And so $f(a)=-j \in f(A)\cap J=(0)$. Consequently, $f(a)=-j=0$ and so $a=0$ since $f$ is injective. It follows that $(a,f(a)+j)=(0,0)$. Hence, $p$ is injective. Thus, $p$ is a ring isomorphism. The conclusion is now straightforward.\\
 Case 2 : Assume that $f(A)\cap J\neq(0)$ and $A$ is reduced.\\
  By $(2)$ above, it remains to show that if $A\bowtie^{f}J$ is Gaussian, then $J^2=0$. Assume that $A\bowtie^{f}J$ is Gaussian. We claim that $J^2=0$. Indeed, let $0\neq f(a)\in f(A)\cap J$ and let $x,y\in J.$ Clearly, $(0,x)$ and $(a,0)$ are elements of $A\bowtie^{f}J$. So, we have $<(a,0),(0,x)>^2=<(a,0)^2>$ or $<(0,x)^2>$. It follows that $x^2=0$ or $a^2=0$. Since $A$ is reduced and $0\neq a$, then $a^2\neq 0$. Hence, $x^2=0$. Likewise $y^2=0$. Therefore, $xy=0$ since $J$ is an ideal of the local Gaussian ring $f(A)+J$ and $<x,y>^2=<x^2,y^2,xy>=<x^2>=0$ or $<y^2>=0$. Hence, $J^2=0$, as desired.\\
      $(4)$ Assume that $f$ is not injective.\\
Case 1 : Suppose that $J\cap Nilp(B)=(0)$ and $A$ is reduced. We claim that $A\bowtie^{f}J$ is not Gaussian. Deny. Using the fact $f$ is not injective, there is some $0\neq a\in Ker(f)$ and so $(a,f(a))=(a,0)\in A\bowtie^{f}J$. Let $x\in J,$ then $(0,x)\in A\bowtie^{f}J$. We have $<(a,0),(0,x)>^2=<(a,0)^2>$ or $<(0,x)^2>$. And so it follows that $x^2=0$ or $a^2=0$. Since $A$ is reduced and $0\neq a$, then $a^2\neq 0$. Hence, $x^2=0$. Therefore, $x\in J\cap Nilp(B)=0$. So, $x=0$. Hence, we obtain $J=0$, a contradiction since $J$ is a proper ideal of $B$. Thus, $A\bowtie^{f}J$ is not Gaussian, as desired.\\
Case 2 : Assume that $J\cap Nilp(B)\neq(0)$ and $A$ is reduced.\\
 By $(2),$ it remains to show that if $A\bowtie^{f}J$ is Gaussian, then $J^2=0$. Suppose that $A\bowtie^{f}J$ is Gaussian. Since $f$ is not injective, then there is some $0\neq a$ $\in Ker(f)$. Let $x,y \in J$, so
$(0,x)$ and $(a,f(a))=(a,0)$ are elements of $A\bowtie^{f}J$. And so, $<(a,0),(0,x)>^2=<(a,0)^2,(0,x)^2>=<(a,0)^2>$ or $<(0,x)^2>$. It follows that $x^2=0$ or $a^2=0$. Since $A$ is reduced and $0\neq a$, then $a^2\neq 0$. Hence, $x^2=0$. Likewise $y^2=0$. Therefore, $xy=0$ since $J$ is an ideal of $f(A)+J$ which is (local) Gaussian by $(1)$ and $<x,y>^2=<x^2,y^2,xy>=<x^2>=0$ or $<y^2>=0$. Hence, $J^2=0,$ as desired.
\end{proof}
The following corollary is a consequence of Theorem \ref{thm0} and is \cite[Theorem 3.2 (2)]{CJKM}.\\
\begin{corollary}\label{cor1}
Let $(A,m)$ be a local ring and $I$ a proper ideal of $A$. Then $A\bowtie I$ is Gaussian if and only if so is $A,I^2=0$ and $aI=a^2I$ for all $a\in m$.
\end{corollary}
\begin{proof}
 It is easy to see that $A\bowtie I=A\bowtie^f J$ where $f$ is the identity map of $A$, $B=A$ and $J=I$. By Lemma \ref{lem1}, $(A\bowtie I,M)$ is local with $M=m\bowtie I$ since $(A,m)$ is local and $I\subseteq Rad(A)=m$. If $A$ is Gaussian, $I^2=0$ and $aI=a^2I$ for all $ a\in m$, then $A\bowtie I$ is Gaussian by $(2)$ of Theorem \ref{thm0}. Conversely, assume that $A\bowtie I$ is Gaussian. By $(2)$ of Theorem \ref{thm0}, it remains to show that $I^2=0.$ Indeed, let $x,y\in I$. Then, $(x,0),(0,x)\in A\bowtie I$ and $<(x,0),(0,x)>^2=<(x,0)^2>$ or $<(0,x)^2>$. Hence, it follows that $x^2=0$. Likewise $y^2=0$. So, $<x,y>^2=<x^2,y^2,xy>=<x^2>=0$ or $<y^2>=0$ since $A$ is (local) Gaussian. Therefore, $xy=0$. Thus, $I^2=0,$ as desired.
  \end{proof}

 The following example illustrates the statement $(3)$ case $2$ of Theorem \ref{thm0}.\\

 \begin{example}\label{sec:2.5}
 Let $(A,m):=(\mathbb{Z}_{(2)},2\mathbb{Z}_{(2)})$ be a valuation domain, $B:=A\propto A$ be the trivial ring extension of $A$ by $A$. Consider $$\begin{array}{clcl}
  f: & A & \hookrightarrow & B \\
   & a & \hookrightarrow & f(a)=(a,0) \\
\end{array}$$ be an injective ring homomorphism and  $J:=2\mathbb{Z}_{(2)}\propto \mathbb{Z}_{(2)}$ be a proper ideal of $B$. Then, $A\bowtie^{f}J$ is not Gaussian.
  \end{example}
   \begin{proof}
 By application to statement $(3)$ case 2 of Theorem \ref{thm0}, $A\bowtie^{f}J$ is not Gaussian since $A$ is reduced and $J^2\neq 0$.
  \end{proof}
Theorem \ref{thm0} enriches the literature with new examples of non-arithmetical Gaussian rings.

 \begin{example}
 Let $(A,m)$ be an arithmetical ring which is not a field such that $m^2=0$ (for instance $(A=\mathbb{Z}/4\mathbb{Z},2\mathbb{Z}/4\mathbb{Z})$), $E$ be a non-zero $\frac{A}{m}-$vector space, $B:=A\propto E$ be the trivial ring extension of $A$ by $E$. Consider $$\begin{array}{clcl}
  f: & A & \hookrightarrow & B \\
   & a & \hookrightarrow & f(a)=(a,0) \\
\end{array}$$ be an injective ring homomorphism and $J:=I\propto E$ be a proper ideal of $B$ with $I$ be a proper ideal of $A$. Then :\\
 $(1)$ $A\bowtie^{f}J$ is Gaussian.\\
 $(2)$  $A\bowtie^{f}J$ is not an arithmetical ring.\\
   \end{example}
   \begin{proof}
 $(1)$ It is easy to see that $J^2=0,$ $f(a)J=f(a)^2J=0$ for all $a\in m$. Hence, by $(2)$ of Theorem \ref{thm0}, $A\bowtie^{f}J$ is Gaussian.\\
  $(2)$ We claim that $A\bowtie^{f}J$ is not an arithmetical ring. Indeed, $f(A)+J=(A\propto 0)+(I\propto E)=A\propto E$ which is not an arithmetical ring (by \cite[Theorem 3.1 (3)]{bkm} since $A$ is not a field). Hence, $A\bowtie^{f}J$ is not an arithmetical ring since the arithmetical property is stable under factor rings.\\
   \end{proof}

 \begin{example}
 Let $(A_0,m_0)$ be a local Gaussian ring, $E$ be a non-zero $\frac{A_0}{m_0}-$vector space, $(A,m):=(A_0\propto E,m_0\propto E)$ be the trivial ring extension of $A_0$ by $E$. Let $E'$ be a $\frac{A}{m}-$vector space, $B:=A\propto E'$ be the trivial ring extension of $A$ by $E'$. Consider $$\begin{array}{clcl}
  f: & A & \hookrightarrow & B \\
   & (a,e) & \hookrightarrow & f((a,e))=((a,e),0) \\
\end{array}$$ be an injective ring homomorphism and $J:=0\propto E'$ be a proper ideal of $B$. Then :\\
 $(1)$ $A\bowtie^{f}J$ is Gaussian.\\
 $(2)$  $A\bowtie^{f}J$ is not an arithmetical ring.\\
   \end{example}
\begin{proof}
$(1)$ By $(3)$ case $1$ of Theorem \ref{thm0}, $A\bowtie^{f}J$ is Gaussian since $f(A)\cap J=(0)$ and $f(A)+J=B$ which is Gaussian by \cite[Theorem 3.1 (2)]{bkm} since $A$ is Gaussian (since $A_0$ is Gaussian). \\
$(2)$  $A\bowtie^{f}J$ is not an arithmetical ring since $f(A)+J=B$ is not an arithmetical ring (by \cite[Theorem 3.1 (3)]{bkm}, $A$ is never a field). \\
 \end{proof}

\begin{example}
Let $(A_0,m)$ be a non-arithmetical Gaussian local ring such that $m^2\neq m,$ for instance $(A_0,m):=(K[[X]]\propto (\frac{K[[X]]}{(X)})^2,XK[[X]]\propto (\frac{K[[X]]}{(X)})^2)$ (by \cite[Theorem 3.1 (2) and (3) ]{bkm}). Consider $A:=A_0\propto \frac{A_0}{m^2}$ be the trivial ring extension of $A_0$ by $\frac{A_0}{m^2}.$ Let $I:=0\propto \frac{m}{m^2}$ be an ideal of $A,$ $B:=\frac{A}{0\propto \frac{m}{m^2}}\cong \frac{A_0\propto \frac{A_0}{m^2}}{0\propto \frac{m}{m^2}}\cong A_0\propto \frac{A_0}{m}$ be a ring, $f :A \rightarrow B$ be a non-injective ring homomorphism and $J:=0\propto \frac{A_0}{m}$ be a proper ideal of $B$. Then:\\
$(1)$ $A\bowtie^{f}J$ is Gaussian.\\
$(2)$ $A\bowtie^{f}J$ is not an arithmetical ring.\\
 \end{example}
  \begin{proof}
$(1)$ It is easy to check that $f(a)J=0$ for all $a\in m\propto \frac{A_0}{m^2}$ which is the maximal ideal of $A$ and $J\subseteq Nilp(B)$. And so  $f(a)^2J=f(a)J=0$ for all $a\in m\propto \frac{A_0}{m^2},$ $J^2=0$ and by \cite[Theorem 3.1 (2)]{bkm}, $A$ is Gaussian since $A_0$ is Gaussian. Hence, by application to the statement $(2)$ of Theorem \ref{thm0}, $A\bowtie^{f}J$ is Gaussian.\\
$(2)$ $A\bowtie^{f}J$ is not an arithmetical ring since $A$ is not an arithmetical ring (since $A_0$ is not an arithmetical ring and so by \cite[Lemma 2.2]{bkm}, $A$ is not an arithmetical ring).
\end{proof}

We need the following result to construct a new class of non-Gaussian Pr\"ufer rings.
\begin{proposition}\label{prop2}
Let $(A,m)$ be a local total ring of quotients, $B$ be a ring, $f : A \rightarrow B$ be a ring homomorphism and  $J$ be a proper ideal of $B$ such that $J\subseteq Rad(B)$ and $J\subseteq Z(B)$. Then, the following statements hold: \\
$(1)$ Assume that $f$ is injective and $f(A)\cap J\neq (0)$. Then $(A\bowtie^{f}J,m\bowtie^{f}J)$ is a local total ring of quotients; In particular, $A\bowtie^{f}J$ is a Pr\"ufer ring.\\
 $(2)$ Assume that $f$ is not injective. Then, $(A\bowtie^{f}J,m\bowtie^{f}J)$ is a local total ring of quotients; In particular, $A\bowtie^{f}J$ is a Pr\"ufer ring.
\end{proposition}
\begin{proof}
By Lemma \ref{lem1}, it is clear that $(A\bowtie^{f}J,m\bowtie^{f}J)$ is local. \\
$(1)$ Assume that $f(A)\cap J\neq (0)$. We claim that $A\bowtie^{f}J$ is a total ring of quotients. Indeed, let $(a,f(a)+j)\in$ $A\bowtie^{f}J,$
 we prove that $(a,f(a)+j)$ is invertible or zero-divisor element. If $a\not \in m$, then $(a,f(a)+j)\not \in m\bowtie^{f}J$. And so $(a,f(a)+j)$ is invertible in $A\bowtie^{f}J$. Assume that $a\in m.$ So, $(a,f(a)+j)\in m\bowtie^{f}J$. Since $A$ is a total ring of quotients, there exists $0\neq b\in A$ such that $ab=0$. We have $(a,f(a)+j)(b,f(b))=(0,jf(b))$. Using the fact $f(A)\cap J\neq (0)$ and $J\subseteq Z(B),$ there exists some $0\neq f(c)\in J$ and $0\neq k\in J$ such that $jk=0$ and so $(c,k)\in A\bowtie^{f}J$. It follows that $(a,f(a)+j)(bc,f(b)k)=(0,0).$ Hence, there exists $(0,0)\neq (bc,f(b)k)\in A\bowtie^{f}J$ such that $(a,f(a)+j)(bc,f(b)k)=(0,0)$. Thus, $(A\bowtie^{f}J,m\bowtie^{f}J)$ is local total ring of quotients.

$(2)$ Assume that $f$ is not injective. Our aim is to show that $A\bowtie^{f}J$ is a total ring of quotients. We prove that for each element $(a,f(a)+j)$ of $A\bowtie^{f}J$ is invertible or zero-divisor element. Indeed, if $a\not \in m$, then $(a,f(a)+j)\not \in m\bowtie^{f}J$. And so $(a,f(a)+j)$ is invertible in $A\bowtie^{f}J$. Assume that $a\in m.$ So, $(a,f(a)+j)\in m\bowtie^{f}J$. Since $A$ is a total ring of quotients, there exists $0\neq b\in A$ such that $ab=0$. We have $(a,f(a)+j)(b,f(b))=(0,jf(b))$. Using the fact $f$ is not injective and $J\subseteq Z(B)$, there exist some $0\neq c\in \Ker(f)$ and $0\neq k\in J$ such that $jk=0$ and $(c,k)\in A\bowtie^{f}J$. It follows that $(a,f(a)+j)(bc,f(b)k)=(0,0)$. Hence, there exists $(0,0)\neq (bc,f(b)k)\in A\bowtie^{f}J$ such that $(a,f(a)+j)(bc,f(b)k)=(0,0)$. Thus, $(A\bowtie^{f}J,m\bowtie^{f}J)$ is a local total ring of quotients, completing the proof.
\end{proof}

Also, Theorem \ref{thm0} enriches the literature with new class of non-Gaussian Pr\"ufer rings.\\
\begin{example}
Let $(A,m)$ be a local total ring of quotients and $I$ be a proper ideal of $A$ such that $I^2\neq 0$. Then :
\begin{enumerate}
\item $A\bowtie I$ is Pr\"ufer.
\item $A\bowtie I$ is not Gaussian.
\end{enumerate}
 \end{example}
\begin{proof}
$(1)$ By $(1)$ of Proposition \ref{prop2}, $A\bowtie I$ is a local total ring of quotients since $A$ is a local total ring of quotients, $A\cap I\neq (0)$ and $I\subseteq m \subseteq Z(A).$\\
$(2)$ By Corollary \ref{cor1}, $A\bowtie I$ is not Gaussian since $I^2\neq 0$.
 \end{proof}

 \begin{example}
Let $(A_0,m_0):=(\mathbb{Z}/2^n\mathbb{Z},2\mathbb{Z}/2^n\mathbb{Z})$ where $n\geq2$ be an integer and let $(A,m):=(A_0\propto A_0,m_0\propto A_0)$ be the trivial ring extension of $A_0$ by $A_0$. Consider $E$ be a non-zero $A$-module such that $mE=0,$ $B:=A\propto E$ be the trivial ring extension of $A$ by $E,$ $$\begin{array}{clcl}
  f: & A & \hookrightarrow & B \\
   & (a,e) & \hookrightarrow & f((a,e))=((a,e),0) \\
\end{array}$$ be an injective ring homomorphism and $J:=m\propto E$ be a proper ideal of $B$. Then, the following statements hold :
\begin{enumerate}
\item $A\bowtie^{f}J$ is Pr\"ufer.
\item $A\bowtie^{f}J$ is not Gaussian.
\end{enumerate}
 \end{example}
\begin{proof}
$(1)$ By $(1)$ of Proposition \ref{prop2}, $A\bowtie^{f}J$ is a local total ring of quotients since $f(A)\cap J=m\propto 0\neq (0),$ $A$ is a local total ring of quotients and $J\subseteq Z(B)$. In particular, $A\bowtie^{f}J$ is a Pr\"ufer ring.\\
$(2)$ $A\bowtie^{f}J$ is not Gaussian since $A$ is not Gaussian (by \cite[Example 3.6]{bmm}).
\end{proof}

\begin{example}
Let $K$ be a field and let $(A_0,m):=(K[[X,Y]],<X,Y>)$ be the ring of formal power series where $X$ and $Y$ are two indeterminate elements. Consider $A:=A_0\propto \frac{A_0}{m^2}$ be the trivial ring extension of $A_0$ by $\frac{A_0}{m^2}.$ Note that $I:=0\propto \frac{m}{m^2}$ is an ideal of $A$. Let $B:=\frac{A}{I}$ be a ring, $f :A \rightarrow B$ be a non-injective ring homomorphism and $J:=\frac{0\propto \frac{A_0}{m^2}}{I}$ be a proper ideal of $B$. Then:\\
$(1)$ $A\bowtie^{f}J$ is $Pr\ddot{u}fer$.\\
$(2)$ $A\bowtie^{f}J$ is not Gaussian.\\
 \end{example}
 \begin{proof}
 $(1)$ One can easily check that $A$ is a local total ring of quotients, $B:=\frac{A}{0\propto \frac{m}{m^2}}\cong A\propto \frac{A}{m},$ $J:=\frac{0\propto \frac{A}{m^2}}{0\propto \frac{m}{m^2}}\cong 0\propto \frac{A}{m}$ and $J\subseteq Z(B)$. Moreover, $J\subseteq Rad(B)=m\propto \frac{A}{m}$ since $B$ is local with maximal ideal $m\propto \frac{A}{m}.$ Hence, by $(2)$ of Proposition \ref{prop2}, $A\bowtie^{f}J$ is local total ring of quotients. Thus, $A\bowtie^{f}J$ is $Pr\ddot{u}fer$.\\
 $(2)$ $A\bowtie^{f}J$ is not Gaussian since $A$ is not Gaussian (since $A_0:=K[[X,Y]]$ is a domain such that $w.dim(K[[X,Y]])=2)$.\\
\end{proof}

\bibliographystyle{amsplain}

\begin{thebibliography}{10}
\bibitem{abjk} J. Abuihlail, M. Jarrar and S. Kabbaj, \textit{Commutative rings in which every finitely generated ideal is quasiprojective},
J. Pure Appl. Algebra 215 (2011) 2504-2511.
\bibitem{and1} D.D. Anderson, \textit{GCD domains, Gauss' Lemma, and contents of polynomials}, Non-Noetherian commutative ring theory, Math. Appl. Kluwer Acad. Publ., Dordrecht 520 (2000) 1-31.
\bibitem{and3} D.D. Anderson and V. Camillo, \textit{Armendariz rings and Gaussian rings}, Comm. Algebra 26 (1998) 2265-2272.
\bibitem{bkm} C. Bakkari, S. Kabbaj and N. Mahdou, \textit{Trivial extensions defined by $Pr\ddot{u}fer$ conditions}, J. of Pure Appl. Algebra 214 (2010) 53-60.
\bibitem{bmm}C. Bakkari, N. Mahdou and H. Mouanis, \textit{Pr\"ufer-like Conditions in Subrings Retract and Applications}, Comm. Algebra 37 (2009) 47-55.
\bibitem{BM} C. Bakkari and N. Mahdou, \textit{Gaussian polynomials and content ideal in pullbacks}, Comm. Algebra 34 (2006) 2727-2732.
\bibitem{bm1}C. Bakkari and N. Mahdou, \textit{Pr\"ufer-like conditions in pullbacks}, Commutative algebra and its applications, Walter de Gruyter, Berlin, (2009) 41-47.
\bibitem{BG} S. Bazzoni and S. Glaz, \textit{Pr\"ufer rings}, Multiplicative Ideal Theory in Commutative Algebra, Springer, New York, (2006) 55-72.
\bibitem{BG2} S. Bazzoni and S. Glaz, \textit{Gaussian properties of total rings of quotients}, J. Algebra 310 (2007) 180-193.
\bibitem{bois} M. Boisen and P. Sheldon, \textit{A note on pre-arithmetical rings}, Acta. Math. Acad. Sci. Hungar 28 (1976) 257-259.
\bibitem{Boisen}M. Boisen and P.B. Sheldon, \textit{CPI-extensions : overrings of integral domains with special prime spectrums}, Canad. J. Math. 29 (1977) 722-737.
\bibitem{BS} H. S. Butts and W. Smith, \textit{Pr\"ufer rings}, Math. Z. 95 (1967) 196-211.
\bibitem{CJKM} M. Chhiti, M. Jarrar, S. Kabbaj and N. Mahdou, \textit{$Pr\ddot{u}fer$ conditions in an amalgameted duplication
of a ring along an ideal}, Comm. Algebra, Accepted for
publication.
\bibitem{c}A. Corso and S. Glaz, \textit{Gaussian ideals and the Dedekind-Mertens Lemma}, Lecture Notes Pure Appl. Math. 217 (2001) 131-143.
\bibitem{AFF1} M. D'Anna, C. A. Finocchiaro and M. Fontana, \textit{Amalgamated algebras along an ideal},
Commutative algebra and its applications, Walter de Gruyter, Berlin, (2009) 241-252.
 \bibitem{AFF2} M. D'Anna, C. A. Finocchiaro and M. Fontana, \textit{Properties
 of chains of prime ideals in amalgamated algebras along an ideal}, J. Pure Appl. Algebra 214 (2010) 1633-1641.
\bibitem{A} M. D'Anna, \textit{A construction of Gorenstein rings}, J. Algebra 306 (2006) 507-519.
\bibitem{AF1} M. D'Anna and M. Fontana, \textit{The amalgamated duplication of a ring along a multiplicative-canonical ideal}, Ark. Mat. 45 (2007) 241-252.
\bibitem{AF2} M. D'Anna and M. Fontana, \textit{An amalgamated duplication of a ring along an ideal : the basic properties}, J. Algebra Appl. 6 (2007) 443-459.
\bibitem{Fu} L. Fuchs, \textit{Uber die Ideale arithmetischer Ringe}, Comment. Math. Helv. 23 (1949) 334-341.
\bibitem{G0} S. Glaz, \textit{Commutative coherent rings}, Multiplicative ideal theory in commutative algebra, Springer, New York, (2006) 55-72.
\bibitem{G2} S. Glaz, \textit{The weak global dimension of Gaussian rings}, Proc. Amer. Math. Soc. 133 (2005) 2507-2513.
\bibitem{G3} S. Glaz, \textit{Pr\"ufer conditions in rings with zero-divisors}, CRC Press Series of Lectures in Pure Appl. Math. 241 (2005) 272-282.
\bibitem{GV} S. Glaz and W. Vasconcelos, \textit{The content of Gaussian polynomials}, J. Algebra 202 (1998) 1-9.
\bibitem{gl} S. Glaz and W. Vasconcelos, \textit{Gaussian polynomials}, Lecture Notes in Pure and Appl. Math. 185 (1997) 325-337.
\bibitem{Gr} M. Griffin, \textit{Pr\"ufer rings with zero-divisors}, J. Reine Angew Math. 239/240 (1969) 55-67.
\bibitem{HH} W. Heinzer and C. Huneke, \textit{Gaussian polynomials and content ideals}, Proc. Amer. Math. Soc. 125 (1997) 739-745.
\bibitem{H} J. A. Huckaba, \textit{Commutative Rings with Zero-Divisors}, Marcel Dekker, New York, 1988.
\bibitem{LR} K. A. Loper and M. Roitman, \textit{The content of a Gaussian polynomial is invertible}, Proc. Amer. Math. Soc. 133 (2004) 1267-1271.
\bibitem{Lu} T. G. Lucas, \textit{Gaussian polynomials and invertibility}, Proc. Amer. Math. Soc. 133 (2005) 1881-1886.
%\bibitem{NN}  B. Nashier and W. Nichols, A note on perfect rings, Manuscripta Math. 70 (3) (1991) 307--310.\par
\bibitem{Nagata} M. Nagata, \textit{Local Rings}, Interscience, New York, 1962.
\bibitem{P} H. Pr\"ufer, \textit{Untersuchungen uber teilbarkeitseigenschaften in korpern}, J. Reine Angew. Math. 168 (1932) 1-36.
%\bibitem{Ra}    W. Rant, Minimally generated modules, Canad. Math. Bull. 23 (1980) 103--105.\par
%\bibitem{Sh}    R. Y. Sharp, Steps in Commutative Algebra. Second edition. London Mathematical Society Student Texts, 51. Cambridge University Press, Cambridge, 2000.\par
\bibitem{T} H. Tsang, \textit{Gauss' Lemma}, Ph.D. thesis, University of Chicago, Chicago, 1965.
\end{thebibliography}

\end{document}